\documentclass[11pt]{article}

\usepackage[a4paper]{anysize}\marginsize{3.5cm}{3.5cm}{1.3cm}{2.5cm}
\pdfpagewidth=\paperwidth \pdfpageheight=\paperheight

\usepackage[latin1]{inputenc}
\usepackage{amssymb}
\usepackage{hyperref}
\usepackage{enumitem}

\makeatletter\renewcommand\@seccntformat[1]{\csname the#1\endcsname.\enspace}\makeatother
\renewenvironment{abstract}{\begin{quote}\hrulefill\par\footnotesize\textbf{\abstractname.}}{\par\vskip-0.5\baselineskip\hrulefill\end{quote}}

\newtheorem{introtheorem}{Theorem}  

\newtheorem{thm}{Theorem}[section]
\newtheorem{theorem}[thm]{Theorem}
\newtheorem{lemma}[thm]{Lemma}
\newtheorem{proposition}[thm]{Proposition}
\newtheorem{corollary}[thm]{Corollary}

\newcommand\mkthm[2]{\newenvironment{#1}{\begin{#2}\rm}{\end{#2}}}
 \mkthm{definition}{tdefinition}
         \mkthm{remark}{tremark}
       \mkthm{example}{texample}
     \mkthm{question}{tquestion}

\newtheorem{thevarthm}[thm]{\varthmname}

\newenvironment{varthm*}[1]{\trivlist\item[]{\bf #1.}\it}{\endtrivlist}

\newenvironment{proof}[1][Proof]{\trivlist\item[\hskip\labelsep{\textit{#1.}}]}{\hspace*{\fill}$\Box$\endtrivlist}

\let\hat=\widehat

\renewcommand\bar{\overline}

\renewcommand\ge{\geqslant}  
\renewcommand\le{\leqslant}  

\newcommand\keywords[1]{{\renewcommand\thefootnote{}\footnotetext{\textit{Keywords:} #1.}}}
\newcommand\subclass[1]{{\renewcommand\thefootnote{}\footnotetext{\textit{Mathematics Subject Classification (2010):} #1.}}}

\newcommand\R{\mathbb R}
\newcommand\Z{\mathbb Z}
\newcommand\N{\mathbb N}

\newcommand\be[1][@{\;}r@{\;}c@{\;}l@{\;}l@{\;}]{$$\everymath{\displaystyle}\renewcommand\arraystretch{1.2}\begin{array}{#1}}
\newcommand\ee{\end{array}$$}

\newcommand\compact{\itemsep=0cm \parskip=0cm}
\newcommand\set[1]{\left\{#1\right\}}
\newcommand\with{\,\,\vrule\,\,}

\newcommand\newop[2]{\newcommand#1{\mathop{\rm #2}\nolimits}}
\newcommand\renewop[2]{\renewcommand#1{\mathop{\rm #2}\nolimits}}

\newop\Diag{Diag}
\newop\Hom{Hom}
\newop\GL{GL}
\newop\LCM{LCM}
\newop\Vol{Vol}
\newop\Pos{Pos}
\newop\Amp{Amp}
\newop\Bigcone{Big}
\newop\Nef{Nef}
\newop\NS{NS}
\newop\Fix{Fix}
\renewop\Re{Re}
\newop\mult{mult}
\newop\SL{SL}
\newop\M{M}
\newop\Div{Div}

\newop\End{End}
\newcommand\Endsym{\End^{\rm sym}}

\newcommand\eps{\varepsilon}

\begin{document}

   \title{Integrality of Seshadri constants and irreducibility of
   principal polarizations on products of two isogenous elliptic curves}
   \author{\normalsize Maximilian Schmidt}
   \date{\normalsize \today}
   \maketitle
   \thispagestyle{empty}
   \subclass{%
      14C20, 
      14H40, 
      14H52, 
      14J99, 
      14K02, 
      14K99, 
      11H55. 
   }

\begin{abstract}
In this paper we consider the question of when all Seshadri constants on a product of two isogenous elliptic curves $E_1\times E_2$ without complex multiplication are integers.
By studying elliptic curves on $E_1\times E_2$ we translate this question into a purely numerical problem expressed by quadratic forms.
By solving that problem, we show that
all Seshadri constants on $E_1\times E_2$ are integers if and only if the minimal degree of an isogeny $E_1\to E_2$ equals 1 or 2.
Furthermore, this method enables a characterization of irreducible principal polarizations on $E_1\times E_2$.
\end{abstract}


\section*{Introduction}
For an ample line bundle $L$ on a smooth projective variety $X$, the \emph{Seshadri constant} of $L$ at a point $x\in X$ is by definition the real number
   \be
      \eps(L,x)=\inf\set{\frac{L\cdot C}{\mult_x(C)}\with C\mbox{ irreducible curve through } x}
      \,.
   \ee
On abelian varieties, $\eps(L,x)$ is independent of the chosen point $x$.
Moreover, one knows by \cite{Bauer-Szemberg:sesh-abelian-surface} that
on abelian \emph{surfaces}
Seshadri constants are always rational numbers.
In the present paper we focus on the
question of when all Seshadri constants $\eps(L)$
on a given abelian surface are \emph{integers}.
This question was first approached in \cite{BGS:integer}, where it was shown
that integrality of Seshadri constants on an abelian surface $X$ is equivalent to requiring that
for every ample line bundle $L$ on $X$, either $\sqrt{L^2}$ is an integer
or $\eps(L)$ is computed by an elliptic curve.
From work of Bauer and Schulz \cite{Bauer-Schulz:sesh} one knows that
the latter condition is always satisfied when
$X$ is a self-product $E\times E$ of an elliptic curve without complex
multiplication.
In the present paper we study the more general situation that
$X$ is a product $E_1\times E_2$ of two isogenous elliptic curves
without complex multiplication.
One might hope that the Seshadri constants on these surfaces behave in a similar way as those
on a self-product, since the surfaces $E_1\times E_1$ and $E_1\times E_2$ are isogenous.
As our result will show, however, Seshadri constants exhibit quite unexpected behaviour under isogenies.

A first indication of this phenomenon follows from work of Kani \cite{Kani:elliptic}, who studied
the question of when $E_1\times E_2$ is the Jacobian of a smooth genus $2$ curve, i.e.,
when $E_1\times E_2$ carries an irreducible principal polarization.
His result states that
$E_1\times E_2$ is \textit{not} a Jacobian if and only if
the minimal degree $d$ of an isogeny $E_1\to E_2$ satisfies
$$d\in \{1, 2, 4, 6, 10, 12, 18, 22, 28, 30, 42, 58, 60, 70, 78, 102, 130, 190, 210, 330, 462\}\,$$
and for at most one more unknown value $d>462$.
Since the Seshadri constant of an irreducible principal polarization is $\frac43$ by \cite{Steffens:remarks},
it follows that on $E_1\times E_2$ non-integer Seshadri constants occur in any event when
$d$ is not contained in the list above.
Our result shows that non-integer Seshadri constants are even more frequent:

\begin{introtheorem}\label{introthm-1}
All Seshadri constants on $E_1\times E_2$ are integers if and only if
the minimal degree $d$ of an isogeny $E_1\to E_2$
equals $1$ or $2$.
\end{introtheorem}
In particular, the non-integer Seshadri constants predicted by the theorem
whenever $d\ge 3$
must come from polarizations other than principal polarizations whenever $d$ is
contained in the list above.

Our proof is based on the
idea to
study the intersection numbers of line bundles with elliptic curves on $E_1\times E_2$, which
allows us to
rephrase the problem into a purely numerical question in terms of quadratic forms.
We will solve that numerical problem by using reduction theory of binary quadratic forms.
This approach also enables us to characterize irreducible principal polarizations on $E_1\times E_2$
in terms of quadratic forms.
Recall that two polarizations $L_1$ and $L_2$ are equivalent if there exists an automorphism $\psi$ of $X$ such that $\psi^*L_1$
is algebraically equivalent to $L_2$,
and denote
the set of isomorphism classes of principal polarizations on $X=E_1\times E_2$ by $P(X)$.
We will see in Sect.~\ref{class-PP} that there is a bijection between $P(X)$ and a set of certain quadratic forms:
\be
	P(X)\cong \left\{
	M=\left(
         \begin{array}{@{}cc@{}}
            A & B   \\
            B & C
         \end{array}
		\right)\in M_2(\Z)
	\,\vrule\,
\begin{array}{l}
   0\le 2B\le A\le C, \\ \gcd(A,B,C)=1,\ \det(M)=d
\end{array}
	\right\}\,.
\ee
We can characterize irreducible principal polarizations in these terms:

\begin{introtheorem}\label{introthm-2}
A matrix $M$ as above corresponds to a class of reducible principal polarizations if and only if $B=0$.
\end{introtheorem}
In particular, there exists an irreducible principal polarization on $E_1\times E_2$ if and only if $d$ can be written as
\be
	d=AC-B^2,\qquad 0\le 2B\le A\le C,\quad \gcd(A,B,C)=1\,.
\ee
As an application we can use Thm.~\ref{introthm-2} to give an alternative proof of Kani's result mentioned above.

Throughout this paper we will work over the complex numbers and on the abelian surface $X=E_1\times E_2$, and we denote by $\varphi:E_1\to E_2$ an isogeny of smallest degree~$d$.


\section{Preliminaries}

Let $E_1$ and $E_2$ be two isogenous elliptic curves without complex multiplication.
Throughout this paper we will work on the abelian surface $X:=E_1\times E_2$ and we fix an isogeny $\varphi:E_1\to E_2$ of smallest degree $d$.
On the product surface $X$, denote by $F_1=\{0\}\times E_2$  and $F_2=E_1 \times \{0\}$ the fibers of the projections and by $\Delta$ the graph of the isogeny $\varphi$.
The classes of these three elliptic curves on $X$ form a basis of the N\'eron--Severi group $\NS(X)$ (see \cite[Thm.~22]{Weil:var-ab} or \cite[Thm.~11.5.1]{Birkenhake-Lange:CAV}).
\begin{proposition}\label{intersection-matrix}
The intersection matrix of $(F_1,F_2,\Delta)$ on $X$ is given by
\be
	\left(
	\begin{array}{ccc}
		0 & 1 & 1  \\
		1 & 0 & d  \\
		1 & d & 0 
	\end{array}
	\right)
	\,.
\ee
\end{proposition}
Note that the determinant of the intersection matrix coincides with the discriminant of the N\'eron--Severi group on $X$ (see \cite[Chp.~3]{Bauer-Born:Cone}).
Thus implies that $(F_1,F_2,\Delta)$ is a basis of $\NS(X)$.
\begin{proof}
We argue as in \cite[Prop.~2.1]{BGS:integer}.
All three curves are elliptic, so we have
$F_1^2=F_2^2=\Delta^2=0$.
As each curve intersects the other ones transversely, it is enough to count the number of intersection points.
So we have
\be
	F_1\cdot F_2=F_1\cdot \Delta=1\,,
\ee
since these curves intersect only in the origin.
For $F_2$ and $\Delta$ one has
\be
	F_2\cdot\Delta=\#\{(x,0)\,|\,x\in E_1\}\cap\{(x,\varphi (x))\,|\,x\in E_1\}\,,
\ee
and this shows that we need to count the number of solutions $x\in E_1$ of the equation $\varphi (x) = 0$.
But this number is equal to the degree of the isogeny $\varphi$, so we get
\be 
	F_2\cdot\Delta=\deg(\varphi)=d\,.
\ee
\end{proof}

As a first application, we will compute Seshadri constants for line bundles in the cone generated by $(F_1,F_2,\Delta)$:
\begin{proposition}\label{sesh-pos-coeff} 
Let $L$ be a line bundle on $X$ of the form
\be
	L=\mathcal O_{E_1\times E_2}(a_1F_1+a_2F_2+a_3\Delta)
\ee
with non-negative coefficients $a_i$.
Then one has
\be
      \eps(L) & = & \min\set{L\cdot F_1, L\cdot F_2, L\cdot \Delta} \\
              & = & \min\{a_2+a_3, a_1+da_3, a_1+da_2\} \,.
\ee
\end{proposition}

\begin{proof}
   Let $D$ be the divisor $a_1F_1+a_2F_2+a_3\Delta$,
   and let $C$ be
   any irreducible curve $C$ passing through $0$,
   which is not a component of $D$.
   As $D$ is effective, we have
   \be
      \frac{L\cdot C}{\mult_0 C}=\frac{D\cdot C}{\mult_0 C} \ge \frac{\mult_0 D\cdot \mult_0 C}{\mult_0 C} &\ge& a_1+a_2+a_3 \\
      &\ge& a_2+a_3 = L\cdot F_1
      \,.
   \ee
This implies that $\eps(L)$ is computed by one of the curves $F_1, F_2$ or $\Delta$.
Their intersection numbers with $L$ are given by the intersection matrix in Prop.~\ref{intersection-matrix} and this yields the assertion.
\end{proof}

\begin{remark}
Let $\hat\varphi:E_2\to E_1$ be the isogeny corresponding to $\varphi$ such that the maps $\varphi\circ\hat\varphi$ and $\hat\varphi\circ \varphi$ are the multiplication by $d$.
Then the previous arguments can be used to see that the triple $(F_1,F_2,\hat\Delta)$ with $\hat\Delta:=\{ x\in X \,|\, x_1=\hat\varphi(x_2)\}$ also forms a basis of $\NS(X)$.
Then, we can analogously formulate Prop.~\ref{sesh-pos-coeff} for line bundles in the cone generated by $(F_1,F_2,\bar\Delta)$.
In general the resulting cones will not coincide since it is possible to get negative coefficients by changing bases:
\be
	F_1+\Delta\equiv dF_1 +(1-d)F_2+\bar\Delta\,,\\
	F_2+\bar\Delta\equiv (1-d)F_1 +dF_2+\Delta\,.
\ee
In fact, the cones coincide if and only if $d=1$, that is, if $E_1$ and $E_2$ are isomorphic.
\end{remark}
For our purpose it will be useful to change the basis $(F_1,F_2,\Delta)$ of the N\'eron--Severi group by choosing an element which is orthogonal to $F_1$ and $F_2$.
We define $\nabla:=\Delta-dF_1-F_2$, then the triple $(F_1,F_2,\nabla)$ forms a basis of $\NS(X)$ and the intersection matrix is given by
\be
		\left(
		\begin{array}{ccc}
			0 & 1 & 0 \\
			1 & 0 & 0 \\
			0 & 0 & -2d  
		\end{array}
		\right)
		\,.
\ee

\begin{proposition}\label{ample-lb}
Consider a line bundle $L=\mathcal O_{X}(a_1F_1+a_2F_2+a_3\nabla)$.
Then $L$ is ample if and only if the two inequalities
\be
	0 &<& a_1\,, \\
	0 &<& a_1a_2-da_3^2=\frac{L^2}{2}\,,
\ee
are satisfied.
\end{proposition}

\begin{proof}
By the Nakai-Moishezon criterion for abelian varieties \cite[Cor.~4.3.3]{Birkenhake-Lange:CAV} a line bundle $L$ is ample if and only if both $L^2= 2(a_1a_2-da_3^2)$ and the intersection of $L$ with any fixed ample line bundle $L_0$ are positive.
By choosing the ample line bundle $L_0:=\mathcal{O}_{E\times E}(F_1+F_2)$ the intersection of $L$ with $L_0$ is given by $a_1+a_2$.
Now assume that $L^2$ is positive, then $a_1$ and $a_2$ are either both positive or both negative. Thus, the sum $a_1+a_2$ is positive if and only if $a_1$ is positive.
\end{proof}


\section{Numerical classes of elliptic curves on $X$}

In this section we will determine the intersection number of line bundles with elliptic curves.
To this end, we need to know all elliptic curves on $E_1\times E_2$.
In \cite{Hayashida-Nishi:genus-two} Hayashida and Nishi described the elliptic curves on $E^n$ and with the same argument it immediately follows:

\begin{lemma}\label{Haya-Nishi}
For every elliptic curve $N$ on $X=E_1\times E_2$ there are endomorphisms $\sigma_1\in\End(E_1)\cong\Z$ and $\sigma_2\in\End(E_2)\cong\Z$ such that $N$ is a translate of the image of the map
\be
	E_1\to E_1\times E_2,\qquad x\mapsto (\sigma_1(x), \sigma_2\circ\varphi (x)).
\ee
\end{lemma}
By the previous lemma we know that for every elliptic curve $N$ on $X$, there are integers $(a,b)\neq (0,0)$ such that $N$ is a translate of 
\be 
	N_{a,b}:=\{(ax,b\varphi(x)\mid x\in E_1\}\,.
\ee
Note that for any multiple $(\lambda a,\lambda b)$ we have $N_{a,b}=N_{\lambda a, \lambda b}=m_\lambda(N_{a,b})$, where $m_\lambda$ describes the multiplication with $\lambda$, which is a map of degree $\lambda^2$ that maps $N_{a,b}$ surjective onto itself.

Next, we want to describe the intersection numbers of elliptic curves with line bundles.
For this, we have to determine the numerical class of the elliptic curves $N_{a,b}\,$.
As $N_{a,b}$ and $F_1$ intersect transversely for $a\neq 0$, we have
\be 
	N_{a,b}\cdot F_1 = \#(N_{a,b}\cap F_1)=\frac{\#\{(x\in E_1 \mid ax=0\}}{\deg(\sigma_{a,b})}\,,
\ee
where $\sigma_{a,b}:E_1\to N_{a,b}$ is the map $x\mapsto (ax,b\varphi(x))$.
But the numerator is equal to the degree of the map $a:E_1\to E_1\,, x\mapsto ax$, which has degree $a^2$.
If $a=0$, then $N_{a,b}=F_1$ and, therefore, $N_{a,b}\cdot F_1=0=\frac{a^2}{\deg(\sigma_{a,b})}$. 
With the same arguments we obtain
\be 
	N_{a,b}\cdot F_1 = \frac{a^2}{\deg(\sigma_{a,b})}\,,  \qquad   N_{a,b}\cdot F_2 = \frac{b^2d}{\deg(\sigma_{a,b})}\,,\qquad
	N_{a,b}\cdot \Delta = \frac{(b-a)^2d}{\deg(\sigma_{a,b})}\,,
\ee
and since $\nabla=\Delta-dF_1-F_2$ we have
\be
	N_{a,b}\cdot \nabla &= \frac{-2abd}{\deg(\sigma_{a,b})}\,.
\ee
The coefficients of the numerical representation of $N_{a,b}$ are then given by
\be 
	\frac{1}{\deg(\sigma_{a,b})}\left(
		\begin{array}{ccc}
			0 & 1 & 0 \\
			1 & 0 & 0 \\
			0 & 0 & -2d  
		\end{array}
	\right)^{-1}
	\cdot \left(
		\begin{array}{c}
            a^2 \\
            b^2d \\
            -2abd 
		\end{array}
      	\right)
      	=\frac{(b^2d,a^2,ab)^T}{\deg(\sigma_{a,b})}\,,
\ee
and, thus, we have
\be 
	N_{a,b}\equiv \frac{b^2d}{\deg(\sigma_{a,b})}F_1 + \frac{a^2}{\deg(\sigma_{a,b})}F_2 + \frac{ab}{\deg(\sigma_{a,b})}\nabla\,.
\ee

To fully understand the intersection numbers we have to determine the degree of the map $\sigma_{a,b}$.
For this, we use the following criteria to identify elliptic curves on abelian varieties (see \cite[Prop.~2.1, Prop.~2.3]{Kani:elliptic}):

\begin{proposition}[\cite{Kani:elliptic}]\label{crit-elli-1}
Let $A$ be an abelian surface and $C_1,C_2\subset A$ be two irreducible curves on $A$.
Then $C_1\cdot C_2\ge 0$.
Furthermore, $C_1\cdot C_2 =0$ if and only if $C_1$ is an elliptic curve and $C_2$ a translate of $C_1$.
\end{proposition}

\begin{proposition}[\cite{Kani:elliptic}]\label{crit-elli-2}
Let $A$ be an abelian surface and $D\in\Div(A)$ a divisor on $A$.
Then $D\equiv mE$ for some $m\in\Z$ and some elliptic curve $E\subset A$ if and only if $D^2=0$.
Moreover, $m>0$ if and only if $\theta\cdot D>0$ for an ample divisor $\theta\in\Div(A)$.
\end{proposition}
With these criteria we can calculate the degree of $\sigma_{a,b}$ and, thus, the numerical class of elliptic curves.

\begin{lemma}\label{numclass-elli}
Let $N\subset X$ be an elliptic curve and let $\sigma_{a,b}:E_1\to N_{a,b}$ be a map $x\mapsto (ax,b\varphi(x))$ such that the image $N_{a,b}$ is a translate of $N$.
Then we have 
\be 
	\deg(\sigma_{a,b})=\gcd(a^2,b^2d,ab)
\ee
and, therefore, the numerical class of $N$ is given by
\be 
	N\equiv N_{a,b}\equiv \frac{b^2dF_1 + a^2F_2 + ab\nabla}{\gcd(a^2,b^2d,ab)}\,.
\ee
\end{lemma}

\begin{proof}
By our previous arguments the numerical class of $N_{a,b}$ is given by
\be 
	N_{a,b}\equiv \frac{b^2d}{\deg(\sigma_{a,b})}F_1 + \frac{a^2}{\deg(\sigma_{a,b})}F_2 + \frac{ab}{\deg(\sigma_{a,b})}\nabla\,.
\ee
Since all the coefficients of the representation must be integers, it follows that $\deg(\sigma_{a,b})$ is a divisor of $b^2d, a^2$ and $ab$. 
Thus, $\deg(\sigma_{a,b})$ is a divisor of $\gcd(a^2,b^2d,ab)$.
So, there exists a positive integer $\lambda$ such that 
\be 
	\lambda\deg(\sigma_{a,b})=\gcd(a^2,b^2d,ab)\,.
\ee
If we show that $\lambda=1$, then the assertion follows.
By using the previous equation in the expression for $N_{a,b}$ we get
\be 
	N_{a,b}\equiv \lambda\left(\frac{b^2d}{\gcd(a^2,b^2d,ab)}F_1 + \frac{a^2}{\gcd(a^2,b^2d,ab)}F_2 + \frac{ab}{\gcd(a^2,b^2d,ab)}\nabla\right)\equiv\lambda \bar{N}\,,
\ee
where $\bar{N}$ is an element of $\NS(X)$, since all the coefficients are integers.

Next, we make use of Prop.~\ref{crit-elli-2} to show that $\bar{N}$ is a positive multiple of an elliptic curve.
Since $N_{a,b}$ is an elliptic curve by assumption, we can apply Prop.~\ref{crit-elli-2} to $N_{a,b}$ and by using $\bar{N}\equiv \frac{1}{\lambda}N_{a,b}$ we see that
\be 
	\bar{N}^2=\frac{1}{\lambda^2}N_{a,b}^2=0
\ee 
and
\be 
	\bar{N}\cdot (F_1+F_2)=\frac{1}{\lambda} N_{a,b}\cdot (F_1+F_2)>0\,.
\ee 
Thus, by Prop.~\ref{crit-elli-2} we have $\bar{N}\equiv mE$ for an elliptic curve $E$ and $m\ge 1$.
So, we obtain $N_{a,b}\equiv m \lambda E$.

Since $N_{a,b}\cdot E = \frac{1}{m\lambda} N_{a,b}^2=0$, it follows from Prop~\ref{crit-elli-1} that $E$ is a translate of $N_{a,b}$ and, therefore, their numerical classes coincide. 
But since $m$ and $\lambda$ are positive integers, it follows that $m=\lambda =1$.
So, we have $\deg(\sigma_{a,b})=\gcd(a^2,b^2d,ab)$ as claimed.
\end{proof}

The natural question might arise whether it is possible to find for any given pair $(a,b)$ another pair $(a',b')$ such that the elliptic curves $N_{a,b}$ and $N_{a',b'}$ coincide and such that the degree of the map $\sigma_{a',b'}:E_1\to N_{a',b'}$ is $1$.
While this is apparently true for $d=1$ by dividing $a$ and $b$ by $\gcd(a,b)$, it is not true for $d>1$, since it would imply that there exists a map of degree $1$, i.e. an isomorphism, of $E_1$ to $N_{0,1}\cong E_2$.

Now we have available the required tools to calculate the intersection numbers of elliptic curves with line bundles:

\begin{proposition}\label{intersection-elli-lb}
	Let $L$ be a line bundle with $L\equiv a_1F_1+a_2F_2+a_3\nabla$ and $N_{a,b}$ an elliptic curve on $X$.
	Then the intersection number of $L$ and $N_{a,b}$ is given by
	\be
		L\cdot N_{a,b}= \frac{1}{\gcd(a^2,db^2,ab)}
		\left(
		\begin{array}{c}
            a \\
            b 
		\end{array}
      	\right)^T
      	\left(
         \begin{array}{cc}
            a_1 & -da_3   \\
            -da_3 & da_2
         \end{array}
		\right)
      	\left(
		\begin{array}{c}
            a    \\
            b 
		\end{array}
		\right)
		\,.
	\ee
\end{proposition}

\begin{proof}
This follows by explicitly calculating the intersection of the numerical class of $N_{a,b}$ given in Lemma \ref{numclass-elli} with the numerical class of $L$.
\end{proof}

\begin{remark}\label{NS-Matrix}
We denote the matrix associated with $L\equiv a_1F_1+a_2F_2+a_3\nabla\in\NS(X)$ through Prop.~\ref{intersection-elli-lb} by
\be 
M_L:=\left(
         \begin{array}{cc}
            a_1 & -da_3   \\
            -da_3 & da_2
         \end{array}
		\right)\,.
\ee
Note that the map $L\mapsto M_L$ is an isomorphism between the abelian groups
\be 

\NS(X)=\{a_1F_1+a_2F_2+a_3\nabla\,\mid\, a_i\in\Z\} \cong \left\{\left(
         \begin{array}{cc}
            a_1 & -da_3   \\
            -da_3 & da_2
         \end{array}
		\right)\,\Big|\, a_i\in\Z\right\}\,.

\ee
\end{remark}

It turns out that this representation enables an ampleness criterion for line bundles on $X$:
\begin{corollary}\label{ample-crit}
Let $L\equiv a_1F_1+a_2F_2+a_3\nabla$ be a line bundle and $M_L$ the corresponding matrix.
Then $\det(M_L)=\frac{dL^2}{2}$ and the following are equivalent:
\begin{enumerate}
\item[\rm(i)] $L$ is ample.
\item[\rm(ii)] $M_L$ is positive definite.
\item[\rm(iii)] The intersection $L\cdot E$ is positive for every elliptic curve $E$ on $X$.
\end{enumerate}
\end{corollary}
\begin{proof}
This is an immediate consequence of Prop.~\ref{ample-lb} and Prop.~\ref{intersection-elli-lb}.
\end{proof}


\section{Integrality of Seshadri constants on $X$}

In this section we will provide a complete answer, in terms of $d$, to the question whether there exists a non-integer Seshadri constant on $X$ or not.
For this, we first develop a criterion to check if all Seshadri constants are computed by elliptic curves.
Obviously, if all Seshadri constants are computed by elliptic curves, then all Seshadri constants are integers by definition.
Our starting point is:

\begin{proposition}\label{Sesh-Num}
	The following are equivalent:
	\begin{enumerate}
	\item[\rm(i)] All Seshadri constants on $X$ are computed by elliptic curves.
	\item[\rm(ii)] Every ample line bundle $L$ on $X$ has a weakly submaximal elliptic curve, i.e., there exists an elliptic curve $E$ on $X$ such that $L\cdot N\le \sqrt{L^2}$.
	\item[\rm(iii)] For every positive definite matrix of the form
		\be
			M=\left(
         	\begin{array}{cc}
            	a_1 & -da_3   \\
            	-da_3 & da_2
         	\end{array}
			\right)
			\in\M_2(\Z)
		\ee
		there exists a coprime pair $(a,b)\in\Z^2\setminus\{0\}$ such that the inequality
		\be 
		\frac{1}{\gcd(a
		,d)}
		\left(
		\begin{array}{c}
            a \\
            b 
		\end{array}
      	\right)^T
      	\left(
         	\begin{array}{cc}
            	a_1 & -da_3   \\
            	-da_3 & da_2
         	\end{array}
		\right)
      	\left(
		\begin{array}{c}
            a    \\
            b 
		\end{array}
		\right)
		\le \sqrt{\frac{2\det(M)}{d}}
		\ee
		is satisfied.
	\end{enumerate}
\end{proposition}

\begin{proof}
The implication (i) $\Rightarrow$ (ii) follows from the definition of Seshadri constants and the upper bound $\varepsilon(L)\le\sqrt{L^2}$.
The implication (ii) $\Rightarrow$ (i) is a result from Schulz.
The argument can be found in the proof of \cite[Thm.~4.5]{Bauer-Schulz:sesh}.

The equivalence of (ii) $\iff$ (iii) is a consequence of Prop.~\ref{intersection-elli-lb} and the fact that it is enough to consider coprime pairs $(a,b)$, since we have $L\cdot N_{a,b}=L\cdot N_{\lambda a,\lambda b}$ for any multiple $\lambda\in\N$.
Moreover, if $a$ and $b$ are coprime, then $\gcd(a^2,b^2d,ab)=\gcd(a,d)$. 
\end{proof}

In Prop.~\ref{Sesh-Num} we have translated the question whether all Seshadri constants are computed by elliptic curves into a purely numerical problem.
To progress further, we will use results for binary quadratic forms and apply them to our setting. 
We start by briefly collecting the relevant language and facts.
Recall that a \textit{(binary) quadratic form} $Q:\Z\times\Z\to \Z$ is given by
	\be
		(x,y)\mapsto  Ax^2+Bxy+Cy^2 =
			\left(
			\begin{array}{c}
				x \\
				y 
			\end{array}
    	  	\right)^T
    	  	\left(
			\begin{array}{cc}
				A & B/2   \\
				B/2 & C
			\end{array}
			\right)
     	 	\left(
			\begin{array}{c}
     	       x    \\
      	      y 
			\end{array}
			\right)\,,
	\ee
where $A, B$ and $C$ are integers.
We will denote $Q$ by the triple $(A,B,C)$.
A quadratic form $Q$ is called \textit{primitive} if $\gcd(A,B,C)=1$.
Two quadratic forms $Q$ and $P$ are called \textit{(properly) equivalent} if there exists a matrix $S\in\GL_2(\Z)$ such that $P(x,y)= Q(S(x,y))$ (and, respectively $\det(S)=1$).
It is more common to use proper equivalence with quadratic forms, since the classes of primitive quadratic forms of a fixed determinant form a group.

The crucial ingredient are so-called reduced forms and their properties:

\begin{definition}
A positive definite quadratic form $(A,B,C)$ is called \textit{reduced}, if the integers $A$, $B$ and $C$ satisfy $|B|\le A\le C$ and, if one of the inequalities is not strict, then $B\ge 0$.
\end{definition}

\begin{theorem}[{\cite[Thm.~2.8]{Cox:binary}}]\label{reduced-forms} Every positive definite quadratic form is properly equivalent to a unique reduced form.
Moreover, there exists an effective algorithm to transform a given form into its unique reduced form.
\end{theorem}
It follows that any positive definite quadratic form is \textit{equivalent} to a unique form $(A,B,C)$ with $0\le B\le A \le C$.
This will be used in the last section.

Now, we turn our attention to the two relevant properties which we obtain by using reduced forms:

\begin{proposition}[{\cite[(2.9), (2.12)]{Cox:binary}}]\label{reduced-properties} Let $Q=(A,B,C)$ be a reduced positive definite quadratic form.
Then:
	\begin{enumerate}
		\item[\rm(i)] We have $A\le \sqrt{\frac{4\det(Q)}{3}}$.
		\item[\rm(ii)] The two smallest integers represented by $Q$ with coprime pairs are $Q(1,0)=A$ and $Q(0,1)=C$.
	\end{enumerate}
\end{proposition}

With this, we can refine Prop.~\ref{Sesh-Num} in terms of reduced positive definite forms.

\begin{proposition}\label{Sesh-Num-refined}
	The following is equivalent:
	\begin{enumerate}
	\item[\rm(i)] All Seshadri constants on $X$ are computed by elliptic curves.
	\item[\rm(ii)] For every reduced positive definite quadratic form $Q=(A,B,C)$ and every 
	$S=\left(
		\begin{array}{cc}
			\alpha & \beta   \\
			\gamma & \delta
		\end{array}
	\right)\in\SL_2(\Z)$ there exists a coprime pair $(a,b)\in\Z^2\setminus\{0\}$ such that
	\be
		\frac{Q(a,b)}{\gcd(a\alpha+b\beta,d)}\le \sqrt{\frac{2\det(Q)}{d}} \,.
	\ee
	\end{enumerate}
\end{proposition}

\begin{proof}
To begin with, we will determine how a base change affects the inequality given in Prop.~\ref{Sesh-Num}.
Assume that we have a positive definite matrix $M$ and for $S=\left(
		\begin{array}{cc}
			\alpha & \beta   \\
			\gamma & \delta
		\end{array}
	\right)\in\SL_2(\Z)$ we set $R:=S^T M S$.
Then, by changing coordinates a coprime pair $(x,y)$ corresponds to the coprime pair $(v,w):=S^{-1}(x,y)^T$ and it follows that
\be 
		\frac{1}{\gcd(x
		,d)}
		\left(
		\begin{array}{c}
            x \\
            y 
		\end{array}
      	\right)^T
      	M\,\,
      	\left(
		\begin{array}{c}
            x    \\
            y 
		\end{array}
		\right)
		=\frac{1}{\gcd(v\alpha+w\beta,d)} 
		\left(
		\begin{array}{c}
            v \\
            w 
		\end{array}
      	\right)^T
      	R\,\,
      	\left(
		\begin{array}{c}
            v    \\
            w 
		\end{array}
		\right)\,.
\ee
Thus, the change of basis leads to the inequality stated in the Proposition.

The implication (ii) $\Rightarrow$ (i) follows from the fact that every positive definite matrix appearing in Prop.~\ref{Sesh-Num} is by Thm.~\ref{reduced-forms} equivalent to a unique reduced form.
Thus, if the inequality holds for all reduced quadratic forms and all $S\in\SL_2(\Z)$, then it already holds for every positive definite form.

The implication (i) $\Rightarrow$ (ii) is less apparent, since Prop.~\ref{Sesh-Num} gives a statement on positive definite quadratic forms given by
$$
\left(
		\begin{array}{cc}
			a_1 & -da_3  \\
			-da_3 & da_2
		\end{array}
	\right)\,, \eqno (\ast)
$$
whereas $(ii)$ is a statement about \textit{all} (reduced) positive definite quadratic forms.
Let $Q=(A,B,C)$ be any reduced positive definite quadratic form and $S\in\SL_2(\Z)$.
First, we observe that the inequality in (ii) is invariant by scaling.
Thus, $Q$ and the coprime pair $(a,b)$ satisfy the inequality if and only if it holds for the coprime pair $(a,b)$ and all forms $\lambda Q=(\lambda A,\lambda B,\lambda C)$ for $\lambda>0$, since we have
\be
		\frac{\lambda Q(a,b)}{\gcd(a\alpha+b\beta,d)}
		 \le \sqrt{\frac{2\lambda^2\det(Q)}{d}}\,. 
\ee
Now, by choosing $\lambda=2d$ and applying the base change $S$ we see that the corresponding matrix
\be 
	M=2d\cdot S^T\left(
		\begin{array}{cc}
			A & B/2  \\
			B/2 & C
		\end{array}
	\right)S
\ee
is of the form $(\ast)$.
Therefore, the inequality holds in fact for all reduced forms.
\end{proof}

As a consequence of Thm.~\ref{reduced-forms} and Prop.~\ref{Sesh-Num-refined} the minimal intersection number with elliptic curves can be effectively computed:

\begin{proposition}\label{algo}
Let $L$ be a line bundle on $X$.
Then there exists an effective algorithm to calculate $\varepsilon^*(L):=\min\{L.N\mid N\subset X \mbox{ elliptic curve}\}$.
\end{proposition}

\begin{proof}
Let $M_L$ be the corresponding matrix to the line bundle $L$.
By Thm.~\ref{reduced-forms} there exists an algorithm which computes the reduced form $Q=(A,B,C)$ and the base change $S$ such that $M_L=S^TQS$.
The claim follows from the fact that there exists only a finite number of pairs $(a,b)$ such that $Q(a,b)\le M$ for any upper bound $M$ and by choosing the upper bound $M=dA$.
\end{proof}

We now turn to the integrality  question.
If $d=1$, then by \cite[Thm.~2.2]{Bauer-Schulz:sesh} every line bundle has a weakly submaximal elliptic curve and, thus, all Seshadri constants are integers.
We give now a complete answer depending on the degree $d$ whether all Seshadri constants are integers or if there exists rational Seshadri constants.

\begin{theorem}\label{Sesh-integer-d12}
If $d=1$ or $d=2$, then all Seshadri constants are integers.
Moreover, all Seshadri constants on $X$ are computed by elliptic curves.
\end{theorem}
The Theorem implies that $\varepsilon^*(L)=\varepsilon(L)$ and hence $\varepsilon(L)$ can be effectively computed by Prop.~\ref{algo}.

\begin{proof}
We will show that Prop.~\ref{Sesh-Num-refined} applies in this situation.
First note that, since the inequality in Prop.~\ref{Sesh-Num-refined} does not depend on $\gamma$ and $\delta$, it is enough to consider coprime pairs $(\alpha,\beta)$.
Let $Q=(A,B,C)$ be a reduced positive definite quadratic form.
We have to show that there exists a coprime pair $(a,b)$ such that
$$
		\frac{Q(a,b)}{\gcd(a\alpha+b\beta,d)}\le \sqrt{\frac{2\det(Q)}{d}} \,.\eqno (\ast)
$$
First, assume that $d=1$.
Then, by Prop.~\ref{reduced-properties}~(i) it follows that we have for $(1,0)$ the inequality
$$
		Q(1,0)=A\le \sqrt{\frac{4\det(Q)}{3}}\le \sqrt{2\det(Q)} 
$$
and, thus, we have found a coprime pair.

Next, we treat the case $d=2$.
We claim that at least one of the pairs $(1,0),(0,1),(1,\pm 1)$ satisfies the inequality $(\ast)$.
We will discuss three cases dependent on the divisibility of $\alpha$ and $\beta$ by $2$.
Note that $\alpha$ and $\beta$ cannot be both divisible by $2$ at the same time, since they are coprime.

First, we assume that $\alpha\equiv 0$ and $\beta\equiv 1$ modulo $2$.
Then we have for $(a,b)=(1,0)$
\be 
	\frac{Q(1,0)}{\gcd(\alpha,2)}=\frac{A}{2}\,\,{\le}\,\, \sqrt{\frac{\det(Q)}{3}},
\ee
where the last estimate follows from Prop.~\ref{reduced-properties}~(ii).
Thus, $(1,0)$ satisfies $(\ast)$.

Secondly, we treat the case $\alpha\equiv 1$ and $\beta\equiv 0$.
We assume to the contrary that $(1,0)$ and $(0,1)$ do not satisfy $(\ast)$.
So, we have
\be 
	\frac{Q(1,0)}{\gcd(\alpha,2)}=A > \sqrt{\det(Q)}\,,
\ee
and
\be 
	\frac{Q(0,1)}{\gcd(\beta,2)}=\frac{C}{2} > \sqrt{\det(Q)}\,.
\ee
Then, by multiplying both inequalities we get
\be 
	B^2 > 2AC\,,
\ee
but this is impossible, since $Q$ is reduced and, therefore, $|B|\le A\le C$.

Lastly, we treat the case $\alpha\equiv 1$ and $\beta\equiv 1$.
This time, we assume that $(1,0)$ and $(1,\pm 1)$ do not satisfy $(\ast)$.
Thus, we have
\be 
	\frac{Q(1,0)}{\gcd(\alpha,2)}=A > \sqrt{\det(Q)}\,,
\ee
and
\be 
	\frac{Q(1,\pm 1)}{\gcd(\alpha \pm \beta,2)}\ge\frac{A+C-|B|}{2} > \sqrt{\det(Q)}\,.
\ee
Then, after multiplying both inequalities we get
\be 
	2A^2+B^2>2AC+2A|B|\,
\ee
which is impossible, since $Q$ is reduced.
\end{proof}

We turn our attention to proving the converse of Thm.~\ref{Sesh-integer-d12}.
This amounts to finding a line bundle $L$ with a non-integer Seshadri constant.
One way to find such a bundle is by searching for irreducible principal polarizations on $X$, since these have $\eps(L)=\frac43$ by \cite[Prop.~2]{Steffens:remarks}.
The existence and number of irreducible principal polarizations has been studied by a number of authors (e.g.~\cite{Earle:Jacobian}, \cite{Hayashida-Nishi:genus-two}, \cite{IKO:Jacobian}, \cite{Kani:jacobian}, \cite{Lange:PP}).
As mentioned in the introduction, Kani answered this question to a great extent for $E_1\times E_2$:
\begin{theorem}[{\cite[Thm.~5]{Kani:jacobian}}]\label{kani:iPP-thm}
There exists no irreducible principal polarization on $X$ if and only if $d=1$ or if $d$ is an even idoneal number which is not divisible by $8$.
This is the case for
$$d\in \{1, 2, 4, 6, 10, 12, 18, 22, 28, 30, 42, 58, 60, 70, 78, 102, 130, 190, 210, 330, 462\}\,$$
and for at most one more $d^*>462$.
\end{theorem}
For a more thorough discussion of idoneal numbers and the existence of the additional number $d^*$ we refer to \cite{Kani:idoneal} and \cite{Kani:jacobian}.

We will now show that there in fact exists an ample line bundle with a non-integer Seshadri constant for \textit{every} $d\ge 3$.
(Such line bundles cannot be principal polarizations in the cases listed in Thm.~\ref{kani:iPP-thm}.)
For this, we will use Thm.~2 from \cite{BGS:integer}.
So, the assertion follows, if we show that there exists a line bundle $L$ such that $\varepsilon(L)$ is not computed by an elliptic curve and that $\sqrt{L^2}$ is not an integer.

\begin{theorem}\label{Sesh-no-submaximal}
If $d\ge 3$, then there exists a line bundle on $X$, which does not have a weakly submaximal elliptic curve. 
\end{theorem}

\begin{proof}
We will use Prop.~\ref{Sesh-Num-refined}.
So, the issue is to exhibit a reduced positive definite form $Q=(A,B,C)$ and a matrix
	$S=\left(
		\begin{array}{cc}
			\alpha & \beta   \\
			\gamma & \delta
		\end{array}
	\right)\in\SL_2(\Z)$
	such that for every coprime pair $(a,b)$ the inequality
	\be 
		\frac{Q(a,b)}{\gcd(a\alpha+b\beta,d)}
		> \sqrt{\frac{2\det(Q)}{d}}
	\ee
is satisfied.
We will show that the quadratic form $Q=(2,1,d)$ and the matrix
\be
	S=\left(
		\begin{array}{cc}
			1 & \Big\lceil \frac{d}{2} \Big\rceil   \\
			0 & 1
		\end{array}
	\right)\in\SL_2(\Z)
\ee
satisfy for every coprime pair $(a,b)$ the even stronger inequality
	$$
		\frac{Q(a,b)}{\gcd(a+b\lceil \frac{d}{2} \rceil,d)}
		\ge 2 > \sqrt{\frac{8d-1}{2d}}=\sqrt{\frac{2\det(Q)}{d}}\,. \eqno (\ast)
	$$

The idea is as follows:
We begin by exhibiting two properties a coprime pair $(a,b)$ must satisfy, if it contradicts the inequality $(\ast)$.
Then we show that no coprime pair can satisfy both properties at the same time.

First, we observe that it is enough to consider coprime pairs $(a,b)$ with
\be 
	Q(a,b)<2d\,,
\ee 
because any coprime pair $(a,b)$ with $Q(a,b)\ge 2d$ will satisfy the inequality $(\ast)$ since the denominator is at most $d$.

Secondly, we claim that it is enough to consider coprime pairs $(a,b)\neq (1,0)$ such that
\be
	d \,\,\mbox{ divides } \,\, a+b \Big\lceil \frac{d}{2} \Big\rceil\,.
\ee
To this end, we know by Prop.~\ref{reduced-properties} that the two smallest non-zero integers represented by $Q$ with coprime pairs are $Q(1,0)=2$ and $Q(0,1)=d$.
Since the pair $(1,0)$ satisfies the inequality $(\ast)$, we may assume that $(a,b)\neq (1,0)$ and, thus, $Q(a,b)$ is at least $d$.
If such a coprime pair has $\gcd(a+b\lceil \frac{d}{2} \rceil,d)=q<d$, then we conclude
\be
	\frac{Q(a,b)}{\gcd(a+b\lceil \frac{d}{2} \rceil,d) }\ge\frac{d}{q}\ge 2,
\ee
as for any divisor $q$ of $d$ with $q\neq d$ the quotient $\frac{d}{q}$ is at least $2$. 
Consequently, we are left with coprime pairs $(a,b)\neq (1,0)$ such that $\gcd(a+b\lceil \frac{d}{2} \rceil,d)=d$.

Next, we will create a set of pairs which will contain all coprime pairs $(a,b)\neq (1,0)$ such that $Q(a,b)<2d$.
We claim that if $|b|> 1$, then $Q(a,b)$ is at least $2d$.
For this, we consider for a fixed $b$ the function $x\mapsto Q(x,b)$, which has the minimum in $\frac{-b}{4}$. 
It follows that we have for every $x\in\R$
\be
	Q(x,b)\ge b^2\,\frac{8d-1}{8} 
\ee
and for $|b|\ge 2$ we obtain $Q(x,b)\ge 2d$.

Thus, we only have to consider the cases $b=1$ and $b=-1$.
Now, we calculate the range of $a$ depending on $b\in\{-1,1\}$ by estimating the possible values of $a$ such that $Q(a,b)=2a^2+ab+db^2<2d$ holds.
We conclude by solving these two quadratic inequalities that
$$
	|a|\le\frac{1}{4}(1 + \sqrt{8d+1})  \,. \eqno (\ast\ast)
$$
It follows that all coprime pairs $(a,b)\neq (1,0)$ with $Q(a,b)< 2d$ are contained in (the possibly bigger set):
\be
	Z=\{(n,\pm 1)\,\mid\, 0\le n\le \Big\lfloor\frac{1}{4}(1 + \sqrt{8d+1})\Big\rfloor, n\in\N \,\}\,.
\ee 
The crucial point is that if $d\ge 4$, then we can derive the following bounds for $a+b\lceil \frac{d}{2}\rceil$: If $b=1$, then
\be 
	 0 	  < 	\Big\lceil \frac{d}{2} \Big\rceil  	\le 	a+ \Big\lceil \frac{d}{2} \Big\rceil   	\le   \Big\lfloor\frac{1}{4}(1 + \sqrt{8d+1})\Big\rfloor + \Big\lceil \frac{d}{2} \Big\rceil   <  d\,
\ee
and if $b=-1$, then
\be 
	-d 	  <  -	\Big\lceil \frac{d}{2} \Big\rceil  	\le 	a- \Big\lceil \frac{d}{2} \Big\rceil   	\le   \Big\lfloor\frac{1}{4}(1 + \sqrt{8d+1})\Big\rfloor - \Big\lceil \frac{d}{2} \Big\rceil   <  0\,.
\ee
(This is also the moment where the argument hinges on the choice of the matrix $S$.)
As a consequence, the sum $a+b\lceil \frac{d}{2}\rceil$ always lies in the interval $(0,d)$ or in $(-d,0)$ and, therefore, is not divisible by $d$.
So, we conclude that every coprime pair $(a,b)$ satisfies $(\ast)$.

It remains to treat the case $d=3$. 
The inequality $(\ast\ast)$ is not helpful in this case, since $Z$ will then contain a pair $(a,b)$, namely $(1,1)$, with $\gcd(a+b\lceil \frac{3}{2} \rceil,3)=3$, which, however, does not satisfy $6=Q(1,1) < 6$.
Instead, in this case we can calculate all coprime pairs $(a,b)\neq(1,0)$ with $Q(a,b)=2a^2+ab+3b^2<6$ explicitly:
\be
	Z_3=\{(0,1), (1,-1)\}\,.
\ee
We see that none of the pairs $(a,b)\in Z_3$ satisfies $\gcd(a+2b,3)=3$, which completes the proof.
\end{proof}

So far we have found line bundles $L$ whose Seshadri constants are not computed by elliptic curves.
This concludes the first step of the proof of the converse of Thm.~\ref{Sesh-integer-d12}.
It still might be possible that the Seshadri constants are all integers if $\varepsilon(L)=\sqrt{L^2}\in\Z$ (see \cite[Ex.~1.1]{BGS:integer}).
To exclude this, we have to show that $L^2$ is not a square number.
With that we can deduce the even stronger statement:

\begin{theorem}\label{Sesh-non-integer}
If $d\ge 3$, then there exists a line bundle on $X$ with a non-integer Seshadri constant.
\end{theorem}

\begin{proof}
We will apply Thm.~2 from \cite{BGS:integer}.
Thus, the claim follows, if we show that there exists a line bundle $L$ such that $\varepsilon(L)$ is not computed by an elliptic curve and that $\sqrt{L^2}$ is not an integer.
To find such a line bundle, we will use the binary quadratic form $Q=(2,1,d)$ and matrix $S=\left(
		\begin{array}{cc}
			1 & \Big\lceil \frac{d}{2} \Big\rceil   \\
			0 & 1
		\end{array}
	\right)\in\SL_2(\Z)$
given in Thm.~\ref{Sesh-no-submaximal} and apply the arguments from the proof of Prop.~\ref{Sesh-Num-refined}.

So, we consider the scaled quadratic form $Q'=2d Q =(4d,2d,2d^2)$ together with the matrix $S$.
By using the change of basis $S$ we get the positive definite matrix
\be
	M & = & S^T
	\left(
		\begin{array}{cc}
			4d & d   \\
			d & 2d^2
		\end{array}
	\right)
	S=\left(
		\begin{array}{cc}
			4d & d\left(4 \Big\lceil\frac{d}{2}\Big\rceil + 1\right)   \\
			d\left(4 \Big\lceil\frac{d}{2}\Big\rceil + 1\right) & d\left(4\Big\lceil\frac{d}{2}\Big\rceil^2 + 2 \Big\lceil\frac{d}{2}\Big\rceil + 2d\right)
		\end{array}
	\right)\,,
\ee
which by Rmk.~\ref{NS-Matrix} corresponds to the class of the ample line bundle 
\be 
	L=\mathcal O_{X}\left(4dF_1+\left(4\Big\lceil\frac{d}{2}\Big\rceil^2 + 2 \Big\lceil\frac{d}{2}\Big\rceil + 2d\right)F_2-\left(4 \Big\lceil\frac{d}{2}\Big\rceil + 1\right)\nabla\right)\,,
\ee
which has no weakly submaximal elliptic curve by Thm.~\ref{Sesh-no-submaximal}.

Thus, it is left to show that $L^2$ is not a perfect square.
By using Cor.~\ref{ample-crit} we have for the self-intersection
\be 
	L^2=\frac{2\det(M)}{d}=16d^2-2d\,.
\ee
If $d$ is odd, then $16d^2-2d\equiv 2$ modulo $4$ and hence $L^2$ is not a perfect square.
If $d$ is even, then we write $d=2^nd'$ for an odd integer $d'\ge 1$ and $n\ge 1$.
If $n$ is even, then $L^2=2^n(2^{n+4}\,d'^2-2d')$. 
Since $2^n$ is a square number, it is enough to show that $2^{n+4}\,d'^2-2d'$ can not be a perfect square.
This, however, can not happen since it is $2$ modulo $4$.

Lastly, if $n$ is odd, then $L^2=2^{n+1}(2^{n+3}\,d'^2-d')$.
Suppose to the contrary that $L^2$ is a perfect square.
Then there exists an integer $r$ such that $2^{n+3}\,d'^2-d'=r^2$.
We claim that $d'$ already is a perfect square.
Let $p$ be a prime number such that $p^{2n-1}$ is a divisor of $d'$.
It follows that $p^{2n-1}$ must be a divisor of $r^2$ and since $r^2$ is a perfect square it is divisible by $p^{2n}$.
Thus, $d'= 2^{n+3}\,d'^2-r^2$ is divisible by $p^{2n}$ as well and, therefore, $d'$ must be an odd square number.
Then, however, $2^{n+3}\,d'^2-d'\equiv 3$ modulo $4$ and hence $L^2$ cannot be a square number, which is a contradiction.
\end{proof}

\begin{remark}
Note that the application of Thm.~2 from \cite{BGS:integer} does not yield an explicit line bundle with a fractional Seshadri constant.
Thus, the line bundles given in the proof do not necessarily have a non-integer Seshadri constant themselves, but they imply their existence.
For such an example see \cite[Prop.~2.8]{BGS:integer}.
\end{remark}


\section{Classification of Principal Polarizations on $X$}\label{class-PP}

In this section we will characterize the isomorphism classes of reducible and irreducible principal polarization.
We give a brief summary of notation, which can be found in \cite{Lange:PP}.
Two principal polarizations $L_1$ and $L_2$ are \textit{equivalent} if there exists an automorphism $\psi$ of $X$ such that $\psi^*L_1 \equiv L_2$. 
We denote the isomorphism classes of principal polarizations on $X$ by
\be 
	P(X)=\{L\in\NS(X)\,\mid\, L \mbox{ ample }, L^2=2\}\,/\sim\,.
\ee

Recall that we have the equality $\frac{dL^2}{2}=\det(M_L)$ and, thus, by Rmk.~\ref{NS-Matrix} a principal polarization $L$ is given by a positive definite matrix $M_L=\left(
         \begin{array}{cc}
            a_1 & -da_3   \\
            -da_3 & da_2
         \end{array}
		\right)$ with determinant $d$.
Furthermore, since the self-intersection is $2$, it follows that the matrix $M_L$ is primitive, that is $\gcd(a_1,da_2,da_3)=1$.
Thus, the quadratic form $M_L$ is equivalent to a unique reduced form $Q=(A,2B,C)$ with $0\le B\le A \le C$ and $\gcd(A,B,C)=1$; we will call such forms \textit{principally reduced forms of determinant $d$}.
(Note that the matrices $M_L$ in Rmk.~\ref{NS-Matrix} only have integer entries, thus the second entry of $Q$ will be even.)
With arguments from the proof of \cite[Thm.~5.1.]{Lange:PP} the converse is true, that is, for any principally reduced form $Q$ of determinant $d$ there exists a principal polarization $L$ on $X$ such that $M_L$ is equivalent to $Q$.

Moreover, it follows form \cite[Chp.~5.]{Lange:PP} that two principal polarizations $L_1$ and $L_2$ are equivalent if and only if the positive definite forms $M_{L_1}$ and $M_{L_2}$ are equivalent to the same principally reduced form.
Thus, we have a natural bijection between the isomorphism classes of principal polarizations and principally reduced forms:

$$
	P(X)\cong \{\mbox{principally reduced forms of determinant } d\}\,.\eqno (\ast)
$$

Lange studied in \cite{Lange:PP} the number of isomorphism classes of principal polarization by using the the isomorphism $\NS(X)\cong \Endsym(X)$.
Thus, his approach is slightly different from ours.
However, one can show that the matrix $M_L$ induced by the intersection of $L$ with elliptic curves (see Rmk.~\ref{NS-Matrix}) is closely related to the image of $L$ under the natural isomorphism $\Phi:\NS(X)\to \Endsym(X)$ given by the principal polarization $L_0=\mathcal O_{X}\left(F_1+F_2\right)$ \cite[Prop.~5.2.1]{Birkenhake-Lange:CAV}.

Using $(\ast)$, we now characterize the principally reduced forms of reducible principal polarizations $L$, i.e., the polarizations such that $L\equiv E_1'+E_2'$ with elliptic curves $E_i'$ with $E_1'\cdot E_2'=1$.

\begin{theorem}\label{crit-irred-PP}
Let $Q=(A,2B,C)$ be a principally reduced form of determinant $d$.
Then the principal polarizations in the isomorphism class given by $Q$ are reducible principal polarizations if and only if $B=0$.
\end{theorem}
In other words, irreducible principal polarizations correspond to principally reduced forms with $B\neq 0$.
\begin{proof}
First, assume that $B=0$.
Let $S=\left(
		\begin{array}{cc}
			\alpha & \beta   \\
			\gamma & \delta
		\end{array}
	\right)\in\GL_2(\Z)$ be a base change such that $Q$ corresponds to a principal polarization $L$.
It is enough to find an elliptic curve such that the intersection with $L$ is $1$.
It follows for
\be 
	M_L:=S^T \left(
		\begin{array}{cc}
			A & 0   \\
			0 & C
		\end{array}
	\right) S,
\ee
that $d$ divides $m_{2,2}=\beta^2 A+ \delta^2C$ and $m_{1,2}=\alpha\beta A+ \delta\gamma C$.
By using $AC=d$ and the coprime properties of $\alpha, \beta,\gamma$ and $\delta$ resulting from the fact that the determinant of $S$ is $1$, we can deduce that $C$ is a divisor of $\beta$ and $A$ is a divisor of $\delta$.
Therefore, we see that $\gcd(\beta,d)= C$.
Hence, we obtain
\be 
	\frac{Q(0,1)}{\gcd(\beta,d)}= 1.
\ee
Thus, the elliptic curve $N_{-\beta,\delta}$ which corresponds to $S^{-1}(0,1)$ has intersection number $1$ with $L$.
Thus, $L$ is reducible.

Let $L$ be a reducible principal polarization.
The issue is to find a base change $S\in\GL_2(\Z)$ such that $S^TM_LS$ is a diagonal matrix.
Since $L$ is reducible, there exists two elliptic curves $E_1'$ and $E_2'$ on $X$ such that $L\equiv E_1'+E_2'$.
By Lemma~\ref{Haya-Nishi} there are coprime pairs $(\alpha,\beta)$ and $(\gamma,\delta)$ such that $E_1'\equiv N_{\alpha,\beta}$ and $E_2'\equiv N_{\gamma,\delta}$ with $\deg(\sigma_{\alpha,\beta})=\gcd(\alpha,d)$ and $\deg(\sigma_{\gamma,\delta})=\gcd(\gamma,d)$.
We will show that the matrix 
$$S:=\left(
		\begin{array}{cc}
			\alpha & \gamma   \\
			\beta & \delta
		\end{array}
	\right)$$ 
has the required properties.
	
For that we first claim that $\alpha$ and $\gamma$ are coprime.
Let $p$ be any prime number and we factorize $\alpha=p^{r}\alpha'$, $\gamma=p^s\gamma'$ and $d=p^td'$.
By Lemma~\ref{numclass-elli} the numerical classes are given by:
\be 
	N_{\alpha,\beta}\equiv & \frac{d\beta^2 F_1 + \alpha^2 F_2 + \alpha\beta \nabla}{\gcd(\alpha,d)}=\frac{p^td'\beta^2 F_1 + p^{2r}\alpha'^2 F_2 + p^{r}\alpha'\beta \nabla}{\gcd(p^r,p^t)\cdot\gcd(\alpha',d')}\,,\\
	\vspace{-0.3cm}\\
	N_{\gamma,\delta}\equiv & \frac{d\delta^2 F_1 + \gamma^2 F_2 + \gamma\delta \nabla}{\gcd(\gamma,d)}=\frac{p^td'\delta^2 F_1 + p^{2s}\gamma'^2 F_2 + p^s\gamma'\delta \nabla}{\gcd(p^s,p^t)\cdot\gcd(\gamma',d')}\,.
\ee
If $p$ is a common factor of $\alpha$ and $\gamma$, then $1\le r,s$.
Without loss of generality we assume that $1\le r\le s$. 
Then we deduce by explicitly calculating the intersection number $N_{\alpha,\beta}\cdot N_{\gamma,\delta}=1$ that 
\be 
1=\frac{p^{2r+t}}{\gcd(p^r,p^t)\cdot\gcd(p^s,p^t)}\cdot\frac{\alpha'^2\delta^2d'+p^{2(s-r)}\gamma'^2\beta^2d'-2p^{m-n}\alpha'\beta\gamma'\delta d}{\gcd(\alpha',d')\cdot\gcd(\gamma',d')}\,.
\ee
For all $t\ge 0$ both factors are integers and, moreover, the left factor is at least $p$. 
Thus, this equation represents a factorization of $1$, but this is impossible and, therefore, $\alpha$ and $\gamma$ are coprime.

Next, we will show that $S$ has determinant $\pm 1$.
The map $$\sigma_{\alpha,\beta}\times \sigma_{\gamma,\delta}:E_1\times E_1\to E_1'\times E_2'\cong X$$ is an isogeny and we have for the degree $$\deg(\sigma_{\alpha,\beta}\times \sigma_{\gamma,\delta})=\deg(\sigma_{\alpha,\beta})\cdot \deg(\sigma_{\gamma,\delta})=
\gcd(\alpha,d)\cdot\gcd(\gamma,d)=\gcd(\alpha\gamma,d)\,,$$
where the last equality comes from the fact that $\alpha$ and $\gamma$ are coprime.
Thus, the degree of $\sigma_{\alpha,\beta}\times \sigma_{\gamma,\delta}$ is at most $d$.
On the other hand, the minimal degree of any isogeny $E_1\times E_1\to E_1\times E_2$ is at least $d$ and, hence, the $\gcd(\alpha\gamma,d)$ equals $d$.
From the equation given by the intersection $N_{\alpha,\beta}\cdot N_{\gamma,\delta}=1$ we then deduce that 
\be
	1=(\alpha\delta-\beta\gamma)^2=\det(S)^2\,. 
\ee
Lastly, by using explicit calculations we see that
\be 
	S^T M_L S = \left(
		\begin{array}{cc}
			\gcd(\alpha,d) &  0  \\
			0 & \gcd(\gamma,d)
		\end{array}
	\right)\,.
\ee
Note that if $\gcd(\alpha,d)>\gcd(\gamma,d)$, then we swap the roles $E_1'$ and $E_2'$.
This completes the proof.
\end{proof}

To conclude, we briefly return to Kani's result in Thm.~\ref{kani:iPP-thm}.
His proof is based on studying the \textit{refined Humbert invariant}
\be 
	q_L(D)=(D\cdot L)^2 -2D^2,\qquad D\in \Div(A)\,,
\ee
as a positive definite quadratic form on the quotient $\NS(A)/\Z L$.
We provide here an alternative proof of Kani's result, as a consequence of the characterization of irreducible principal polarizations given by Thm.~\ref{crit-irred-PP}.
For this, we use a criterion by Grube \cite[Chp. 9]{Grube:Idoneal} to characterize idoneal numbers, which can be refined as follows \cite[Rmk.~18 (b)]{Kani:idoneal}:
\begin{proposition}[\cite{Grube:Idoneal} \cite{Kani:idoneal}]\label{Grube:crit-idoneal}
An integer $n\ge 1$ is an idoneal number if and only if for every $B=1,...,\lfloor \sqrt{\frac{n}{3}}\rfloor$, and for every representation
\be 
	n + B^2 = AC
\ee
with $2B\le A\le C$ and $\gcd(A,2B,C)=1$, 
we have either $A=C$ or $A=2B$.
\end{proposition}
\begin{proof}[Proof of Thm.~\ref{kani:iPP-thm}]
To begin with, we claim that $d$ is an idoneal number.
By Thm.~\ref{Grube:crit-idoneal} there does not exist an irreducible principal polarization on $X$ if and only if there does not exist a representation $d+B^2=AC$ with $\gcd(A,B,C)=1$ and $0<2B\le A \le C$.
Thus, by the previous criterion \ref{Grube:crit-idoneal} it follows that $d$ is an idoneal number.

To complete the claim, we have to show that for an idoneal number $d$ there does not exist a representation $d+B^2=AC$ with $\gcd(A,B,C)=1$ and $0<2B\le A \le C$ if and only if $d\equiv 2,4,6$ modulo $8$.

To this end, we first show that there exists an irreducible polarization if $d$ is odd or if $d$ is divisible by $8$. 
If $d$ is odd, then we have $(A,B,C)=(2,1,\frac{d+1}{2})$.
If $d$ is divisible by $8$, then we have $(4,2,\frac{d+4}{4})$ for $d>8$ and $(3,1,3)$ for $d=8$.

Lastly, we show that there does not exist an irreducible polarization if $d$ is idoneal and $d\equiv 2,4,6$ modulo $8$.

Case $1$: Let $d$ be an idoneal number with $d\equiv 2$ or $d\equiv 6$, then we have $d=2d'$ for an odd number $d'$.
Assume that we have a representation $d+B^2=AC$ with $\gcd(A,B,C)=1$ and $C\ge A\ge 2B>0$. 
Since $d$ is idoneal, it follows from Prop.~\ref{Grube:crit-idoneal} that $A=2B$ or $A=C$.

Case $1.1$: If $A=2B$, then $2d'+B^2=2BC$ and it follows that $B$ is even.
But then $2d'=2BC-B^2$ is divisible by $4$ which is impossible since $d'$ is odd.

Case $1.2$: If $A=C$, then we have $2d'+B^2=C^2$.
Since $\gcd(A,B,C)=1$, it follows that $B$ and $C$ are odd.
But then $2d'=(C-B)(C+B)$ is divisible by $8$ which is a contradiction.

Case $2$: Let $d$ be an idoneal number with $d\equiv 4$, then we have $d=4d'$ for an odd number $d'$.
Assume that we have a factorization $d+B^2=AC$ with $\gcd(A,B,C)=1$ and $C\ge A\ge 2B>0$. 
Again, we either have $A=2B$ or $A=C$.

Case $2.1$: If $A=2B$, then $4d'+B^2=2BC$ and it follows that $B$ is even.
Thus, there exists an odd number $B'$ and $n\ge 1$ such that $B=2^nB'$.
Since we have $4d'+2^{2n}B'^2=2^{n+1}B'C$, it follows that $n=1$, because otherwise $d'$ would be divisible by $2$.
But then $C$ has to be an even number, since we have $d'+B'=B'C$.
This, however, is impossible, because then we would have $\gcd(A,B,C)\ge 2$.

Case $2.2$: If $A=C$, then we argue as in case $1.2$ and deduce that
$4d'=(C-B)(C+B)$ is divisible by $8$ which is a contradiction since $d'$ is odd.
\end{proof}



\footnotesize
   \bigskip
   Maximilian Schmidt
   Fachbereich Mathematik und Informatik,
   Philipps-Universit\"at Marburg,
   Hans-Meerwein-Stra\ss e,
   D-35032 Marburg, Germany.

   \nopagebreak
   \textit{E-mail address:} \texttt{schmid4d@mathematik.uni-marburg.de}


\end{document}